\documentclass{gtmon_a}
\pdfoutput=1

\usepackage{amscd}

\proceedingstitle{Proceedings of the Nishida Fest (Kinosaki 2003)}
\conferencestart{28 July 2003}
\conferenceend{8 August 2003}
\conferencename{International Conference in Homotopy Theory}
\conferencelocation{Kinosaki, Japan}

\editor{Matthew Ando}
\givenname{Matthew}
\surname{Ando}

\editor{Norihiko Minami}
\givenname{Norihiko}
\surname{Minami}

\editor{Jack Morava}
\givenname{Jack}
\surname{Morava}

\editor{W Stephen Wilson}
\givenname{W Stephen}
\surname{Wilson}

\title[Classifying spaces of compact Lie groups]{Classifying spaces of compact Lie groups that are\\$p$--compact for
all prime numbers}

\author{Kenshi Ishiguro}
\givenname{Kenshi}
\surname{Ishiguro}
\address{Fukuoka University\\\newline
Fukuoka 814-0180\\
Japan}
\email{kenshi@cis.fukuoka-u.ac.jp}
\urladdr{}

\volumenumber{10}
\issuenumber{}
\publicationyear{2007}
\papernumber{11}
\startpage{195}
\endpage{211}

\doi{}
\MR{}
\Zbl{}

\arxivreference{}


\received{3 June 2004}
\revised{14 February 2005}
\accepted{}
\published{29 January 2007}
\publishedonline{29 January 2007}
\proposed{}
\seconded{}
\corresponding{}
\version{}

\makeatletter
\def\cnewtheorem#1[#2]#3{\newtheorem{#1}{#3}[section]
\expandafter\let\csname c@#1\endcsname\c@thm}

\let\xysavmatrix\xymatrix
\def\xymatrix{\disablesubscriptcorrection\xysavmatrix}
\AtBeginDocument{\let\bar\wbar\let\tilde\wtilde\let\hat\what}

\newtheorem{thm}{Theorem}[section]
\newtheorem{bigthm}{Theorem}
\cnewtheorem{prop}[thm]{Proposition}
\cnewtheorem{lem}[thm]{Lemma}
\cnewtheorem{cor}[thm]{Corollary}
\makeautorefname{bigthm}{Theorem}

\makeatother  

\newcommand{\Spin}{\mathit{Spin}}
\newcommand{\Pin}{\mathit{Pin}}
\newcommand{\Sp}{\mathit{Sp}}
\newcommand{\SO}{\mathit{SO}}
\newcommand{\OO}{\mathit{O}}

\makeop{rank}

\def \F{\mathbb{F}}

\def \wt{\widetilde}
\def \2{^{\wedge}_2}
\def \3{^{\wedge}_3}
\def \p{^{\wedge}_p}
\newcommand{\lra}{{\longrightarrow}}
 
\begin{document}

\begin{abstract}
We consider a problem on the conditions of a compact Lie group $G$ that
the loop space of the $p$--completed classifying space be a $p$--compact
group for a set of primes.  In particular, we discuss the classifying
spaces $BG$ that are $p$--compact for all primes when the groups are
certain subgroups of simple Lie groups.  A survey of the $p$--compactness
of $BG$ for a single prime is included.
\end{abstract}

\maketitle

A $p$--compact group (see Dwyer--Wilkerson \cite{DWb}) is a loop space $X$ such that $X$ is 
$\F_p$--finite and that its classifying space $BX$ is $\F_p$--complete
(see Andersen--Grodal--M{\o}ller--Viruel~\cite{AGMV} and Dwyer--Wilkerson~\cite{DWn}).  
We recall that the $p$--completion of a compact Lie
group $G$ is a $p$--compact group if $\pi_0(G)$ is a $p$--group.  Next, if $C(\rho )$ denotes the
centralizer of a group homomorphism $\rho$ from a $p$--toral group to a compact Lie group, 
according to \cite[Theorem~6.1]{DWb},
the loop space of the $p$--completion $\Omega (BC(\rho ))\p$ is a $p$--compact group.

In a previous article \cite{Ig2}, the classifying space $BG$
is said to be {\it $p$--compact} if 
$\Omega (BG)\p$ is a $p$--compact group.  There are some results 
for a special case.  A survey is given in \fullref{sec1}. 
It is well-known that, if
$\Sigma_3$ denotes the symmetric group of order $6$, then
$B\Sigma_3$ is not $3$--compact.  In fact, for a finite group $G$, the classifying space $BG$
is $p$--compact if and only if $G$ is $p$--nilpotent.  
Moreover, we will see that $BG$ is $p$--compact toral (see Ishiguro
\cite{It}) if and
only if the compact Lie group $G$ is $p$--nilpotent (see Henn \cite{He}).  For the general case,
we have no group theoretical characterization, though a few necessary conditions are
available.  This problem is also discussed in the theory of $p$--local
groups (see Broto, Levi and Oliver \cite{BLO,BLOs}) from a different
point of view.

We consider the $p$--compactness of $BG$ for a set of primes.
Let $\Pi$ denote the set of all primes.  For a non--empty subset
$\mathbb{P}$ of 
$\Pi$, we say that $BG$ is {\it $\mathbb{P}$--compact} if this space is $p$--compact
for any $p \in \mathbb{P}$.  If $G$ is connected, then $\Omega (BG)\p \simeq G\p$
for any prime $p$, and hence
$BG$ is $\Pi$--compact.
The connectivity condition, however, is not necessary.  For instance, 
the classifying space of each
orthogonal group $\OO(n)$ is also $\Pi$--compact.
Since $\pi_0(\OO(n))=\Z /2$ is a $2$--group, $B\OO(n)$ is $2$--compact, and for any odd prime $p$,
the $p$--equivalences $B\OO(2m) \simeq_{p} B\OO(2m+1) \simeq_{p}
B\SO(2m+1)$ tell us that $B\OO(n)$ is $\Pi$--compact.

Next let $\mathbb{P} (BG)$ denote the set of primes $p$ such that $BG$ is $p$--compact.
In \cite{It} the author
has determined $\mathbb{P} (BG)$ when $G$ is the normalizer $NT$ of a maximal torus $T$ of a 
connected compact simple Lie group $K$ with Weyl group $W(K)$.  Namely 
$$
\mathbb{P}(BNT) =
\begin{cases}
\Pi
&\text{if } W(K) \text{ is a 2--group,}\\
\{ p \in \Pi~|~|W(K)| \not\equiv 0 \mod p \} &\text{otherwise.}
\end{cases}
$$
Other examples are given by a subgroup $H\cong SU(3)\rtimes \Z/2$ of 
the exceptional Lie group $G_2$ and its quotient group $\Gamma_2 =H/(\Z/3)$.  
$$
\CD
\Z /3 @= \Z /3 @>>> *   \\
@VVV         @VVV  @VVV           \\
SU(3) @>>> H @>>>  \Z /2    \\
@VVV         @VVV  @|        \\
PU(3) @>>> \Gamma_2 @>>> \Z /2 
\endCD
$$   
A result of \cite{Ig2} implies that $\mathbb{P} (BH) =\Pi$ and $\mathbb{P} (B\Gamma_2) =\Pi -\{ 3\}$.

In this paper we explore some necessary and sufficient
conditions for a compact Lie group to be $\Pi$--compact. 
First we consider a special case.  We say that $BG$ is 
{\it $\mathbb{P}$--compact toral} if
for each $p \in \mathbb{P}$ the loop space $\Omega (BG)\p$ is expressed as an
extension of a $p$--compact torus $T\p$ by a finite $p$--group $\pi$
so that that there is a fibration $(BT)\p \longrightarrow (BG)\p \longrightarrow B\pi$.  Obviously, if $BG$ is $\mathbb{P}$--compact toral, the space is $\mathbb{P}$--compact.
A necessary and sufficient condition that $BG$ be $p$--compact toral is given in \cite{It}.  
As an application, we obtain the following:

\begin{bigthm}
\label{thm1}
Suppose $G$ is a compact Lie group, and $G_0$ denotes its connected component with
the identity.  Then $BG$ is $\Pi$--compact toral if and
only if the following two conditions hold:
\begin{enumerate}
\item[\rm(a)] $G_0$ is a torus $T$, and the group $G/G_0=\pi_0G$ is nilpotent.
\item[\rm(b)] $T$ is a central subgroup of $G$.
\end{enumerate}
\end{bigthm}

For a torus $T$ and a finite nilpotent group $\gamma$, the product group $G=T \times \gamma$
satisfies conditions (a) and (b).  Thus $BG$ is $\Pi$--compact toral.
\fullref{prop2.1} will show,
however, that a group $G$ with $BG$ being $\Pi$--compact toral need not be a
product group.

Next we ask if $BH$ is $\mathbb{P}$--compact when $H$ is a subgroup of a simple Lie group $G$.
For $\mathbb{P} =\Pi$, the following result determines certain types of $(G, H_0)$
where $H_0$ is the connected component of the identity.  We have seen the cases of
$(G, H)=(G, NT)$ when $W(G)=NT/T$ is a $2$--group, 
and of $(G, H)=(G_2, SU(3)\rtimes \Z/2)$ which is considered as a case with
$(G, H_0)=(G_2, A_2)$.  Recall that the Lie algebra of $SU(n+1)$ is simple of 
type $A_n$, and the Lie group $SU(3)$ is of $A_2$--type (see
Bourbaki~\cite{Bourbaki}).
 
\begin{bigthm}
\label{thm2}
Suppose a connected compact Lie group $G$ is simple.
Suppose also that $H$ is a proper closed subgroup of $G$
with $\rank(H_0)=\rank(G)$,
and that the map $BH \lra BG$ induced by the inclusion
is $p$--equivalent for some $p$.  Then the following hold:

\begin{enumerate}
\item[\rm(a)] If the space $BH$ is $\Pi$--compact,
$(G, H_0)$ is one of the following types:
\[
(G, H_0)=
\left\{
\begin{array}{llll}
(G, T_G)  &\mbox{for $G=A_1$ \ or \ $B_2(=C_2)$ } \\
(B_n, D_n) \\
(C_2, A_1 \times A_1) \\
(G_2, A_2)
\end{array}
\right.
\]
where $T_G$ is the maximal torus of $G$.

\item[\rm(b)] For any odd prime $p$, all above types are realizable.  Namely,
there are $G$ and $H$ of types as above such that $BH$ is $\Pi$--compact, together with
the $p$--equivalent map $BH \lra BG$.
When $p=2$, any such pair $(G, H)$ is not realizable.
\end{enumerate}
\end{bigthm}

We make a remark about covering groups.  Note that if $\alpha \lra \wt G \lra G$ is
a finite covering, then $\alpha$ is a central subgroup of $ \wt G$.
For a central extension $\alpha \lra \wt G \lra G$ and a subgroup $H$ of $G$, we
consider the following commutative diagram:
$$
\CD
\alpha @>>> \wt G @>>> G \\
 @|  @AAA   @AAA               \\
\alpha @>>> \wt H @>>> H
\endCD
$$
\noindent Obviously the vertical map $H \lra G$ is the inclusion, and $\wt H$ is the induced
subgroup of $ \wt G$.  We will show that the pair $(G, H)$ satisfies the
conditions of \fullref{thm2}
if and only if its cover $(\wt G, \wt H)$ satisfies those of
\fullref{thm2}.  Examples of the type
$(G, H_0)=(B_n, D_n)$, for instance, can be given by
$(\SO(2n+1), \OO(2n))$ and the double cover $(\Spin(2n+1), \Pin(2n))$.  

For the case $(G, H_0)=(G_2, A_2)$, we have seen that $H$ has a finite normal subgroup $\Z/3$,
and that for its quotient group $\Gamma_2$ the classifying space $B\Gamma_2$ is $p$--compact
if and only if $p\ne 3$.  So $\mathbb{P} (B\Gamma_2) \ne \Pi$.  The following result
shows that this is the only case.  Namely, if $\Gamma$ is such a quotient group
for $(G, H_0) \ne (G_2, A_2)$, then $\mathbb{P} (B\Gamma) = \Pi$.

\begin{bigthm}
\label{thm3}
Let $(G, H)$ be a pair of compact Lie groups
as in \fullref{thm2}.  For a finite normal subgroup $\nu$ of $H$,
let $\Gamma$ denote the quotient group $H/\nu$.  If $(G, H_0) \ne (G_2, A_2)$,
then $B\Gamma$ is $\Pi$--compact.
\end{bigthm}

The author would like to thank the referee for the numerous suggestions.

\section{A survey of the $p$--compactness of $BG$ }
\label{sec1}

We summarize work of earlier articles \cite{Ig2,It} together with some
basic results, in order
to introduce the problem of $p$--compactness.
For a compact Lie group $G$, the classifying space $BG$ is $p$--compact
if and only if $\Omega (BG)\p$ is $\F_p$--finite.  So it is a mod $p$ finite H--space.  The space $B\Sigma_3$ is not $p$--compact
for $p=3$.  We notice that $\Omega (B\Sigma_3)\3$ is not a mod $3$ finite H--space, since the degree of the first non--zero homotopy group of 
$\Omega (B\Sigma_3)\3$
is not odd.  Actually there is a fibration $\Omega (B\Sigma_3)\3 \lra (S^3)\3 
\lra (S^3)\3$ (see Bousfield and Kan \cite{BK}).

First we consider whether $BG$ is $p$--compact toral, as a special case.  
When $G$ is finite, this is the same as asking if
$BG$ is $p$--compact.  Note that, for a finite group $\pi$, the classifying space $B\pi$ is
an Eilenberg--MacLane space $K(\pi , 1)$.  Since $(BT)\p$ is also Eilenberg--MacLane,
for $BG$ being $p$--compact toral,
the $n$--th homotopy groups of $(BG)\p$ 
are zero for $n \ge 3$.  A converse to this fact is the following.

\begin{thm}{\rm\cite[Theorem~1]{It}}\qua
Suppose $G$ is a compact Lie group, and $X$ is a $p$--compact group.  Then
we have the following:
\begin{enumerate}
\item[\rm(i)] If there is a positive integer $k$ such that
$\pi_n((BG)\p )=0$ for any $n \ge k$,
then $BG$ is $p$--compact toral.

\item[\rm(ii)] If there is a positive integer $k$ such that
$\pi_n(BX)=0$ for any $n \ge k$,
then $X$ is a $p$--compact toral group.
\end{enumerate}
\end{thm}

This theorem is also a consequence of work of Grodal \cite{Gt,Gs}

A finite group $\gamma$ is $p$--nilpotent if and only if $\gamma$ is
expressed as the semidirect product $\nu \rtimes \gamma_p$,
where $\nu$ is the subgroup generated by all elements of order prime to $p$,
and where $\gamma_p$ is the $p$--Sylow subgroup.  The group $\Sigma_3$ is $p$--nilpotent if and only if $p\ne 3$.
Recall that a fibration of connected spaces
$F \lra E \lra B$
is said to be preserved by the $p$--completion if $F\p \lra E\p \lra B\p$
is again a fibration.  When $\pi_0(G)$ is a $p$--group, a result of
Bousfield and Kan \cite{BK} implies that
the fibration $BG_0 \lra BG \lra B\pi_0G$ is preserved by the $p$--completion, and $BG$ is $p$--compact.

We have the following necessary and sufficient conditions that $BG$ be $p$--compact toral.
\begin{thm}{\rm\cite[Theorem~2]{It}}\qua
Suppose $G$ is a compact Lie group, and $G_0$ is the connected component with
the identity.  Then $BG$ is $p$--compact toral if and only if the following conditions hold:
\begin{enumerate}
\item[\rm(a)] $G_0$ is a torus $T$ and $G/G_0=\pi_0G$ is $p$--nilpotent.

\item[\rm(b)] The fibration $BT \lra BG \lra B\pi_0G$ is preserved by
the $p$--completion.
\end{enumerate}
Moreover, the $p$--completed fibration
$(BT)\p \lra (BG)\p \lra (B\pi_0G)\p$ splits if and only if
$T$ is a central subgroup of $G$.
\end{thm}

Next we consider the general case.  What are the conditions that $BG$ be $p$--compact?  For example,
for the normalizer $NT$ of a maximal torus $T$ of a connected compact Lie group $K$, it is well--known
that $(BNT)\p \simeq (BK)\p$ if $p$ does not divide the order of the Weyl group $W(K)$.  This means that 
$BNT$ is $p$--compact for such $p$.  Using the following result, we can show the converse.

\begin{prop}{\rm\cite[Proposition~3.1]{It}}\qua
If $BG$ is $p$--compact, then the following hold:
\begin{enumerate}
\item[\rm(a)] $\pi_0G$ is $p$--nilpotent.

\item[\rm(b)] $\pi_1((BG)\p )$ is isomorphic to a $p$--Sylow subgroup
of $\pi_0G$.
\end{enumerate}
\end{prop}

The necessary condition of this proposition is not sufficient, even though the
rational
cohomology of $(BG)\p$ is assumed to be expressed as a ring of invariants under the action of a group generated by pseudoreflections.

\begin{thm}{\rm\cite[Theorem~1]{Ig2}}\qua
\label{thm04}
Let $G=\Gamma_2$, the quotient group of a subgroup $SU(3)\rtimes \Z/2$ of 
the exceptional Lie group $G_2$.  For $p=3$, the following hold:
\begin{enumerate}
\item[\rm(1)] $\pi_0G$ is $p$--nilpotent and
$\pi_1((BG)\p )$ is isomorphic to a $p$--Sylow subgroup of $\pi_0G$.

\item[\rm(2)] $(BG)\p$ is rationally equivalent to $(BG_2)\p$.

\item[\rm(3)] $BG$ is not $p$--compact.
\end{enumerate}
\end{thm}

We discuss invariant rings and some properties of $B\Gamma_2$ and $BG_2$ at $p=3$.  Suppose $G$ is 
a compact connected Lie group.  The Weyl group $W(G)$ acts on its maximal torus $T^n$, and the integral 
representation $W(G) \lra GL(n, \Z )$ is obtained (see Dwyer and
Wilkerson~\cite{DWgs,DWc}).  It is well--known
that $K(BG)\cong K(BT^n)^{W(G)}$ and $H^*(BG; \F_p)\cong H^*(BT^n; \F_p)^{W(G)}$ for large $p$.  Let
$W(G)^*$ denote the dual representation of $W(G)$.  Although the mod 3 reductions of the integral representations
of $W(G_2)$ and $W(G_2)^*$ are not equivalent, there is $\psi \in GL(2, \Z )$ such that 
$\psi W(G_2) \psi^{-1} = W(G_2)^*$ \cite[Lemma~3]{Ig2}.  Consequently,
$K(BT^2 ; \Z \3)^{W(G_2)} \cong K(BT^2 ; \Z \3)^{W(G_2)^*} $.  Since 
$K(B\Gamma_2 ; \Z \3) \cong K(BT^2 ; \Z \3)^{W(G_2)^*} $, we have the following result.

\begin{thm}{\rm\cite[Theorem~3]{Ig2}}\qua
Let $\Gamma_2$ be the compact Lie group as in \fullref{thm04}. 
Then the following hold:
\begin{enumerate}
\item[\rm(1)] The $3$--adic K-theory $K(B\Gamma_2 ; \Z \3)$ is isomorphic
to $K(BG_2 ; \Z \3)$as a $\lambda$--ring.
\item[\rm(2)] Let $\Gamma$ be a compact Lie group such that $\Gamma_0=PU(3)$ and the order of $\pi_0(\Gamma)$
is not divisible by $3$.
Then any map from $(B\Gamma)\3$ to $(BG_2)\3$ is null homotopic.  In particular
$[(B\Gamma_2)\3, (BG_2)\3]=0$.
\end{enumerate}
\end{thm}

We recall that if a connected compact Lie group $G$ is simple, 
the following results hold:

\begin{enumerate}
\item For any prime $p$, the space $(BG)\p$ has no nontrivial retracts
(see Ishiguro~\cite{Isr}).
\item Assume $|W(G)|\equiv 0\ \text{mod} \ p$ .  If a self-map 
$(BG)\p \lra (BG)^p$
is not null homotopic, it is a homotopy equivalence (see
M{\o}ller~\cite{Moeb}).
\item Assume $|W(G)|\equiv 0\ \text{mod} \ p$, and let $K$ be a compact Lie
group. If
a map $f\co (BG)\p \lra (BK)\p$ is trivial in mod $p$ cohomology, then $f$
is null homotopic (see Ishiguro~\cite{Is}).
\end{enumerate}

Replacing $G$ by $\Gamma_2$ at $p=3$,  we will see that (3) still holds.  On the
other hand it is not known if (1) and (2) hold, though on the level of K-theory
they do.

\section{$\Pi$--compact toral groups}

Recall that a finite group $\gamma$ is  {\it $p$--nilpotent} if and only if $\gamma$ is
expressed as the semidirect product $\nu \rtimes \gamma_p$,
where the normal $p$--complement
$\nu$ is the subgroup generated by all elements of order prime to $p$,
and where $\gamma_p$ is the $p$--Sylow subgroup.  For such
a group $\gamma$, we see $(B\gamma)\p \simeq
B\gamma_p$.  For a finite group $G$, one can show that
$\mathbb{P} (BG) =\{ p \in \Pi \ | \ G\ \text{is $p$--nilpotent} \} $.  Consequently, if
$G=\Sigma_n$, the symmetric group on $n$ letters, then $\mathbb{P} (B\Sigma_2) = \Pi$,
$\mathbb{P} (B\Sigma_3) =\Pi -\{ 3\}$, and $\mathbb{P} (B\Sigma_n) =\{ p \in \Pi \ | \ p>n \}$
for $n\ge 4$.

In \cite{He}, Henn provides a generalized
definition of $p$--nilpotence for compact Lie groups.  
A compact Lie group G is \emph{$p$--nilpotent} if and only if the
connected component of the identity, $G_0$, is a torus; the finite group
$\pi_0 G$ is p-nilpotent, and the cojugation action of the normal
$p$--complement is trivial on $T$.  We note that such a $p$--nilpotent group need not be
semidirect product.

Let $\gamma = \pi_0G$.  Then, from the inclusion $\gamma_p \lra \gamma$, a subgroup $G_p$
of $G$ is obtained as follows:
$$
\CD
T @>>> G @>>> \gamma    \\
@|         @AAA   @AAA          \\
T @>>> G_p @>>> \gamma_p 
\endCD 
$$
A result of Henn \cite{He} shows $(BG)\p \simeq (BG_p)\p$ if and
only if the compact Lie group $G$ is $p$--nilpotent.

\begin{lem}
A classifying space $BG$ is $p$--compact toral if and
only if the compact Lie group $G$ is $p$--nilpotent.
\end{lem}

\begin{proof}
If $BG$ is $p$--compact toral, we see from \cite[Theorem~2]{It} that the fibration 
$BT \lra BG \lra B\pi_0G$ is preserved by the $p$--completion.  Let $\pi =\pi_0G$.
Then we obtain the following commutative diagram:
$$
\CD
(BT)\p @>>> (BG)\p @>>> (B\pi)\p   \\
 @|  @AAA   @AAA                  \\
(BT)\p @>>> (BG_p)\p @>>> (B\pi_p)\p
\endCD
$$
By \cite[Theorem~2]{It}, the finite group $\pi$ is $p$--nilpotent, so the map 
$(B\pi_p)\p \lra (B\pi)\p$ is homotopy equivalent.  Thus $(BG)\p \simeq (BG_p)\p$,and hence the result of \cite{He} implies that $G$ is $p$--nilpotent.
Conversely, if $G$ is $p$--nilpotent, then the following commutative diagram
$$
\CD
BT @>>> BG @>>> B\pi  \\
 @|  @AAA   @AAA                  \\
BT @>>> BG_p @>>> B\pi_p
\endCD
$$
tells us that $BT \lra BG \lra B\pi$ is $p$--equivalent to
the fibration
$$(BT)\p \lra (BG_p)\p \lra (B\pi_p)\p.$$
From \cite[Theorem~2]{It}, we see that $BG$ is $p$--compact toral. 
\end{proof}

\begin{proof}[Proof of \fullref{thm1}]
First suppose $BG$ is $\Pi$--compact toral.  Lemma 2.1 implies that
$G_0$ is a torus $T$ and $G/G_0=\pi_0G$ is $p$--nilpotent for any $p$.  
According 
to \cite[Lemma~2.1]{It}, the group $\pi_0G$ must be nilpotent.  
We notice that
for each $p$ the normal $p$--complement of $\pi_0G$ acts trivially on $T$.  
Thus
$\pi_0G$ itself acts trivially on $T$, and $T$ is a central 
subgroup of $G$.
Conversely, assume that conditions (a) and (b) hold.  According 
to \cite[Proposition~1.3]{He}, we see that $G$ is $p$--nilpotent for any $p$.  Therefore
$BG$ is $\Pi$--compact toral.
\end{proof}

We will show that a group which satisfies conditions (a) and (b) of
\fullref{thm1}
need not be a product group.  For instance, consider the quaternion group $Q_8$ in $SU(2)$.
Recall that the group can be presented as $Q_8=\langle x, y\ |\ x^4=1,
x^2=y^2, yxy^{-1}=x^{-1}\rangle$.
Let $\rho\co Q_8 \lra U(2)$ be a faithful representation given by the following:
$$
\begin{matrix}
\rho (x) =\begin{pmatrix} i &0\\
 0&-i\end{pmatrix} &,\ \ \ \rho (y)=\begin{pmatrix} 0&-1\\ 
1&0 \end{pmatrix}
\end{matrix}
$$
\noindent Let $S$ denote the center of
the unitary group $U(2)$ and let $G$ be the subgroup of $U(2)$ generated by
$\rho (Q_8 )$ and $S$.  Then we obtain the group extension
$S \lra G \lra \Z /2 \oplus \Z /2$.  Since $S\cong S^1$, this group $G$
satisfies conditions (a) and (b).  On the other hand, we see that
the non--abelian group $G$ can not be a product group.  This result can be
generalized as follows:

\begin{prop}
\label{prop2.1}
Suppose $\rho \co \pi \lra U(n)$ is a faithful irreducible representation
for a non--abelian finite nilpotent group $\pi$.  Let $S$ be the center of
the unitary group $U(n)$ and let $G$ be the subgroup of $U(n)$ generated by
$\rho (\pi )$ and $S$ with group extension $S \lra G \lra \pi_0G$.  Then
this extension does not split, and $G$ satisfies conditions (a) and (b)
of \fullref{thm1}.
\end{prop}

\begin{proof}
First we show that $G$ satisfies conditions (a) and (b)
of \fullref{thm1}.
Since $\pi$ is nilpotent, so is the finite group $\pi_0G \cong G/S$.  
Recall that the center of the unitary group $U(n)$ consists of scalar matrices, and is
isomorphic to $S^1$.  Thus we obtain the desired result.

Next we show that the group extension $S \lra G \lra \pi_0G$ does not split.
If this extension did split, then we would have $G \cong S \rtimes \pi_0G$.  Since the
action of $\pi_0G$ on the center $S$ is trivial, it follows that $G$ is isomorphic to
the product group $S \times \pi_0G$.  Let $Z(\pi )$ denote the center of $\pi$.
Since the representation $\rho \co \pi \lra U(n)$ is irreducible and
faithful, Schur's Lemma
implies $S\cap \rho (\pi ) =Z(\rho (\pi )) \cong Z(\pi )$.  Thus we obtain
the following commutative diagram:
$$
\CD
S @>>> G @>>> \pi_0G    \\
 @AAA   @AAA    @|              \\
Z(\pi ) @>>> \pi @>q >> \pi_0G
\endCD
$$
Regarding $\pi$ as a subgroup of $G=S \times \pi_0G$, an element $y\in \pi$ can be
written as $y=(s, x)$ for $s\in S$ and $x\in \pi_0G$.  Notice that $\pi_0G$ is nilpotent
and this group has a non--trivial center, since $\pi$ is non--abelian.  
The map $q \co \pi \lra \pi_0G$ is an epimorphism.  Consequently we can find an element $y_0=(s_0, x_0)$
where $s_0 \in S$ and $x_0$ is a non--identity element of $Z(\pi_0G )$.  This means
that $y_0$ is contained in $Z(\pi )$, though $q(y_0)$ is a non--identity element.  This 
contradiction completes the proof. 
\end{proof}

\section{$\Pi$--compact subgroups of simple Lie groups }

We will need the following results to prove \fullref{thm2}.

\begin{lem}
\label{lem3.1}
Let $K$ be a compact Lie group, and let $G$ be a connected compact Lie group.  
If $(BK)\p \simeq (BG)\p$ for some $p$, we have a group extension as follows:
$$1\lra W(K_0)\lra W(G) \lra \pi_0K \lra 1$$
\end{lem}

\begin{proof}
It is well--known that
$H^*((BG)\p ; \Q) = H^*((BT_G)\p ; \Q)^{W(G)}$, and
since $(BK)\p \simeq (BG)\p$, it follows that
$H^*((BG)\p ; \Q) = H^*((BK)\p ; \Q)$.
Notice that
$H^*((BK)\p ; \Q) = \smash{H^*((BK_0)\p ; \Q)^{\pi_0K}}
  = \smash{(H^*((BT_{K_0})\p ; \Q)^{W(K_0)})^{\pi_0K}}$.
Galois theory 
for the invariant rings (see Smith \cite{Si}) tells us that
$W(K_0)$ is a normal subgroup of $W(G)$ and that the quotient group $W(G)/W(K_0)$
is isomorphic to $\pi_0K$.  This completes the proof.  
\end{proof}

\begin{lem}
\label{lem3.2}
For a compact Lie group $K$, suppose
the loop space of the $p$--completion $\Omega (BK)\p$ is a connected $p$--compact group.  Then
$p$ doesn't divide the order of $\pi_0K$.
\end{lem}

\begin{proof}
Since $BK$ is $p$--compact, $\pi_0K$ is $p$--nilpotent.  So, if $\pi$ denotes a $p$--Sylow
subgroup of $\pi_0K$, then $(B\pi_0K)\p \simeq B\pi$.  Notice that $(BK)\p$ is 1--connected.
Hence the map $(BK)\p \lra (B\pi_0K)\p$ induced from the epimorphism $K \lra \pi_0K$
is a null map.  Consequently the $p$--Sylow subgroup $\pi$ must be trivial.
\end{proof}

For $K=NT$, the normalizer of a maximal torus $T$ of a 
connected compact simple Lie group,
the converse of \fullref{lem3.2} is true, though it doesn't hold in general.  Note that $\pi_0\Gamma_2 =\Z/2$
and that $B\Gamma_2$ is not $3$--compact \cite{Ig2}.

\begin{proof}[Proof of \fullref{thm2}]
(1)\qua Since $(BH)\p \simeq (BG)\p$ for some $p$, \fullref{lem3.1} says that the Weyl group $W(H_0)$
is a normal subgroup of $W(G)$.  First we show that $W(H_0) \ne W(G)$.  If
$W(H_0) = W(G)$, the inclusion $H_0 \lra G$ induces the isomorphism
$H^*(BH_0 ; \Q) \cong H^*(BG ; \Q)$, since $\rank(H_0)=\rank(G)$.
Hence $BH_0 \simeq_0 BG$.  Consequently if $\wt{H}_0$ and $\wt{G}$ denote
the universal covering groups of $H_0$ and $G$ respectively, then $\wt{H}_0 \cong \wt{G}$.
The maps $B\wt{H}_0 \lra BH_0$ and $B\wt{G} \lra BG$ are rational
equivalences.  According
to \cite[Lemma~2.2]{Ia}, we would see that $H_0 = H = G$.  Since $H$ must be
a proper subgroup of $G$, we obtain the desired result.

We now see that $W(H_0)$ is a proper normal subgroup of $W(G)$.  If $W(H_0)$ is a
nontrivial group, a result of Asano~\cite{A} implies that $(G, H_0)$ is one of the following types:
$$
(G, H_0)= \begin{cases}
(B_n, D_n) \\
(C_n, A_1 \times \cdots \times A_1) \\
(G_2, A_2) \\
(F_4, D_4)
\end{cases}
$$ 
According to \cite[Lemma~2.1 and Proposition~3.1]{It}, we
notice $\pi_0H = W(G)/W(H_0)$ is a nilpotent group
since $BH$ is $\Pi$--compact.  Recall that 
$W(C_n)/W(A_1 \times \cdots \times A_1) \cong \Sigma_n$ and
$W(F_4)/W(D_4) \cong \Sigma_3$.
For $n \ge 3$, we notice that the symmetric group $\Sigma_n$ is not nilpotent.
Hence the group
$W(G)/W(H_0)$ is nilpotent when $W(G)/W(H_0)  \cong \Z /2$.  Consequently, we see the following:
$$
(G, H_0)=
\left\{
\aligned &(B_n, D_n) \\
&(C_2, A_1 \times A_1) \\
&(G_2, A_2) 
\endaligned
\right.
$$ 
It remains to consider the case that $W(H_0)$ is a trivial group.  In this case
$\pi_0H = W(G)$, and $W(G)$ is a nilpotent group.  From
\cite[Proposition~3.4]{It},
we see that $G=A_1\ \text{or}\ B_2(=C_2)$.

(2)\qua We first show that, for any odd prime $p$, 
all types of the pairs are realized for some $G$ and $H$.  To begin with,
we consider the case $(G, T_G)$ for $G=A_1$.  Take $(G, H)=(\SO(3),
\OO(2))$.  Since
$\pi_0(\OO(2))=\Z /2$ and $B\OO(2) \simeq_p B\SO(3)$ for odd prime $p$, the space
$B\OO(2)$ is $\Pi$--compact.  In the case $G=B_2$, take $(G, H)=(G, NT_G)$
for $G=\Spin(5)$.  Then
$\pi_0H$ is a $2$--group and $BNT_G \simeq_p BG$ for odd prime $p$, and hence
$BNT_G$ is $\Pi$--compact.

In the case of $(B_n, D_n)$, take $(G, H)=(\SO(2n+1), \OO(2n))$.  Since
$\pi_0(\OO(2n))=\Z /2$ and $B\OO(2n) \simeq_p B\SO(2n+1)$ for odd prime $p$, the space
$B\OO(2n)$ is $\Pi$--compact.  For $(C_2, A_1 \times A_1)$, take $G=\Sp(2)$ and
$H=(\Sp(1) \times \Sp(1))\rtimes \Z /2\langle a \rangle$ where 
$a=\bigl(\begin{smallmatrix}  0 &1\\ 1&0 \end{smallmatrix}\bigr)  \in \Sp(2)$.  For complex
numbers $z$ and $w$, we see that
$$
\begin{matrix} 
\begin{pmatrix}  0 &1\\ 1&0 \end{pmatrix}  \begin{pmatrix}  z&0\\ 0&w
\end{pmatrix} 
\begin{pmatrix} 0 &1\\ 1&0\end{pmatrix}
&= \begin{pmatrix} w&0\\ 0& z \end{pmatrix}
\end{matrix}.
$$
\noindent Thus the action of $\Z /2\langle a\rangle$ is given by
$\bigl(\begin{smallmatrix}  0 &1\\ 1&0 \end{smallmatrix}\bigr) $.  We note that
$$ W(\Sp(2))=D_8=\left\langle \begin{matrix} 
\begin{pmatrix}  -1 &0\\ 0&1 \end{pmatrix},
\begin{pmatrix}  1 &0\\ 0&-1 \end{pmatrix}, \begin{pmatrix}  0 &1\\ 1&0 
\end{pmatrix} 
\end{matrix} \right\rangle.
$$
Consequently $\pi_0H$ is a $2$--group and $BH \simeq_p BG$ for odd prime
$p$, and hence $BH$ is $\Pi$--compact.  Finally, for $(G_2, A_2)$, as
mentioned in the introduction, take $G=G_2$ and $H=SU(3)\rtimes \Z/2$.
Then $BH$ is $\Pi$--compact.

It remains to consider the case $p=2$.  Note that $|W(G)/ W(H_0)|$ for
each of such $(G, H_0)$'s is a power of $2$.  \fullref{lem3.1} implies
that the finite group $\pi_0H$ must be a 2--group.  \fullref{lem3.2} says
that $|\pi_0H|$ is not divisible by $2$, since $(BH)\2 \simeq (BG)\2$.
Thus $H$ is connected, and hence $H=G$.  This completes the proof.
\end{proof}

Any proper closed subgroup of $G$ which includes the normalizer $NT$
satisfies the assumption of \fullref{thm2}.  So, this theorem shows,
once again, that almost all $BNT$ are not $\Pi$--compact \cite{It}.
Furthermore, for any connected compact Lie group $G$, it is well-known
that $(BNT)\p \simeq (BG)\p$ if $p$ does not divide the order of the Weyl
group $W(G)$, hence $BNT$ is $p$--compact for such $p$.  The converse
is shown in \cite{It}.

\begin{lem}
\label{lem3.3}
Let $\alpha \lra \wt G \lra G$ be a central extension of compact Lie groups.  
Then $BG$ is $p$--compact if and only if $B\wt G$ is $p$--compact.
\end{lem}

\begin{proof}
First assume that $BG$ is $p$--compact.  
Since $\alpha \lra \wt G \lra G$ is a central extension, the fibration
$B\alpha \lra B\wt G \lra BG$ is principal.
Thus we obtain a fibration $B\wt G \lra BG \lra K(\alpha ,2)$.  
The base space is $1$--connected,
so the fibration is preserved by the $p$--completion, and hence we obtain
the fibration
$$(B\alpha )\p \lra (B\wt G)\p \lra (BG)\p.$$
Since the loop spaces
$\Omega (B\alpha )\p$ and $\Omega (BG)\p$ are $\F_p$--finite, so is
$\smash{\Omega (B\wt G)\p}$.
Thus $\smash{B\wt G}$ is $p$--compact.

Conversely we assume that $B\wt G$ is $p$--compact.  Consider 
the fibration
$$\Omega (BG)\p \lra (B\alpha )\p \lra (B\wt G)\p.$$
Since the map
$(B\alpha )\p \lra (B\wt G)\p$ is induced from the inclusion $\alpha \hookrightarrow \wt G$,
it is a monomorphism of $p$--compact groups.  Hence its homotopy fiber $\Omega (BG)\p$ is
$\F_p$--finite, and therefore $BG$ is $p$--compact.
\end{proof}

\begin{cor}
Let $\alpha \lra \wt G \lra G$ be a central extension of compact Lie groups, and let
$H$ be a subgroup of $G$ so that there is the commutative diagram:
$$
\CD
\alpha @>>> \wt G @>>> G \\
 @|  @AAA   @AAA               \\
\alpha @>>> \wt H @>>> H
\endCD
$$
\noindent Then the pair $(G, H)$ satisfies the conditions of
\fullref{thm2}
if and only if so does the pair $(\wt G, \wt H)$.
\end{cor}

\begin{proof}
\fullref{lem3.3} implies that $BH$ is $\Pi$--compact if and only if
$B\wt H$ is $\Pi$--compact.  It is clear that $\rank(H_0)=\rank(G)$
if and only if $\rank(\wt H_0)=\rank(\wt G)$.  Finally we see
$(BH)\p \simeq (BG)\p$ if and only if $(B\wt H)\p \simeq (B\wt G)\p$ from the following commutative 
diagram of fibrations:
$$
\CD
(B\alpha )\p @>>> (B\wt G)\p @>>> (BG)\p \\
 @|  @AAA   @AAA               \\
(B\alpha )\p @>>> (B\wt H)\p @>>> (BH)\p
\endCD
$$
This completes the proof.
\end{proof}

\begin{lem}
\label{lem3.4}
Let $M \lra K \lra L$ be a short exact sequence of groups.  If $\nu$
is a normal subgroup of $K$, the kernel $\nu '$ of the composition $\nu
\lra K \lra L$ is a normal subgroup of $M$.
\end{lem}

\begin{proof}
We consider the following commutative diagram:
$$
\CD
\nu ' @>>> M \\
  @VVV   @VVV               \\
\nu  @>>> K \\
  @VVV   @VV q V               \\
q(\nu)  @>>> L
\endCD
$$
For $x \in \nu '$ and $m \in M$, it follows that
$$
\begin{array}{rcl}
 q(mxm^{-1})&=q(m)q(x)q(m^{-1})\\
 &=q(m)q(m)^{-1}=e
\end{array}
$$
Thus $mxm^{-1} \in \ker q$.  Since $\nu ' \subset \nu $, $M \subset K$, and
$\nu \vartriangleleft K$, we see that $mxm^{-1} \in \nu$.  So
$mxm^{-1} \in \ker q\cap \nu =\nu '$, and therefore $\nu ' 
\vartriangleleft M$. 
\end{proof}

\begin{proof}[Proof of \fullref{thm3}]
First suppose $(G, H_0)=(B_n, D_n)$ or $(C_2, A_1 \times A_1)$.  Let $\nu '$ be 
the kernel of the composition $\nu \lra H \lra \pi_0H$.  Consider the following commutative diagram:
$$
\CD
\nu ' @>>> H_0 \\
  @VVV   @VVV               \\
\nu  @>>> H \\
  @VVV   @VV q V               \\
q(\nu)  @>>> \pi_0H
\endCD
$$
\fullref{lem3.4} says that $\nu ' \vartriangleleft H_0$.  Since $\nu '$ is a finite normal subgroup of $H_0$, 
it is a finite $2$--group.  As we have seen in the proof of \fullref{thm2}, $\pi_0H = W(G)/W(H_0)$ is a
$2$--group, and hence so is $q(\nu)$.  Consequently $\nu$ is a $2$--group.

Now consider the following commutative diagram:
$$
\CD
\nu ' @>>> \nu @>>> q(\nu)  \\
  @VVV   @VVV  @VVV               \\
H_0  @>>> H @>>> \pi_0H \\
  @VVV   @VVV @VVV              \\
\Gamma_0  @>>> \Gamma @>>> \pi_0 \Gamma
\endCD
$$
Since $\pi_0 \Gamma$ is a $2$--group, the fibration $B\Gamma_0  \lra
B\Gamma \lra B\pi_0 \Gamma$ is preserved by the $2$--completion
(see Bousfield and Kan \cite{BK}).  Hence $B\Gamma$ is $2$--compact.
Next, for odd prime $p$, we see that $(B\Gamma)\p \simeq (BH)\p$,
since $\nu$ is a $2$--group.  We see also that $G$ has no odd torsion
and $H^*(BH ; \F_p) = H^*(BH_0 ; \F_p)^{\pi_0H} \cong H^*(BG ; \F_p)$.
Consequently the space $(B\Gamma)\p$ is homotopy equivalent to $(BG)\p$.
Therefore $B\Gamma$ is $\Pi$--compact.

It remains to consider the case $(G, H_0)=(G, T_G)$ for $G=A_1$ or
$G=B_2(=C_2)$.  Since $H_0=T_G$ and $H_0 \vartriangleleft H$, we see
that $H$ is a subgroup of the normalizer $NT_G$.  Consider the following
commutative diagram:
$$
\CD
T_G @>>> NT_G @>>> W(G) \\
 @|  @AAA   @AAA               \\
T_G @>>> H @>>> \pi_0H
\endCD
$$
Since the map $BH \lra BG$ is $p$--equivalent for some $p$, it follows
that $\pi_0H=W(G)$.  Consequently $H=NT_G$.

If $\nu$ is a finite normal subgroup of $NT_G$, then $B\nu$ is contained
in the kernel of the map $(BG)\p \simeq (BNT_G)\p \lra (B\Gamma)\p$.
Since $G$ is simple and $G\ne G_2$, according to \cite{Is,Ir}, the group
$\nu$ is included in the center of $G$.  Thus $\nu$ is a $2$--group.
Therefore $(B\Gamma)\p \simeq (BNT_G)\p \simeq (BG)\p$ for odd prime $p$,
and hence $B\Gamma$ is $p$--compact for such $p$.  Finally we note that
$W(G)$ is a $2$--group, and hence $B\Gamma$ is $2$--compact.
\end{proof}

We will discuss a few more results.  Basically we have been looking
at three Lie groups $H_0 \subset H \subset G$.  The following shows a
property of the (non--connected) middle group $H$.

\begin{prop}
\label{prop3.1}
Suppose $G$ is a connected compact Lie group, and $H$ is a proper closed subgroup of $G$
with $\rank(H_0)=\rank(G)$.  If the order of $\pi_0H$ is divisible by a prime $p$,
so is the order of $W(G)/W(H_0)$.
\end{prop}

\begin{proof}
Assuming $|W(G)/W(H_0)| \not \equiv 0\ \text{mod}\ p$, we will show 
$\pi_0H \not \equiv 0\ \text{mod}\ p$.  Notice that we have the following commutative diagram
$$
\CD
T @>>> N_GT @>>> W(G) \\
 @|  @AAA   @AAA               \\
T @>>> N_{H_0}T @>>> W(H_0),
\endCD 
$$
where the vertical maps are injective, since
$\rank(H_0)=\rank(G)$.  We recall, from Jackowski, McClure
and Oliver~\cite{JMO}, that the Sylow theorem for compact Lie groups
$G$ holds.  Namely $G$ contains maximal $p$--toral subgroups, and all of
which are conjugate to $N_pT$, where $N_p(T)/T$ is a $p$--Sylow subgroup
of $N(T)/T=W(G)$.

Suppose $K$ is a $p$--toral subgroup of $H$.  Since $|W(G)/W(H_0)| \not \equiv 0\ \text{mod}\ p$,
we see that $K$ is a subgroup of $H_0$ up to conjugate.  Consequently, the composite map
$K \hookrightarrow H \lra \pi_0H$ must be homotopy equivalent to a null map.  Since
$H \lra \pi_0H$ is surjective, the $p$--part of $\pi_0H$ is trivial.
\end{proof}

For each pair mentioned in the part (a) of \fullref{thm2}, we note that $|W(G)/W(H_0)|$ is a power
of $2$.  \fullref{prop3.1} says, for instance, that $\pi_0H$ is a $2$--group for any $(G, H)$
such that $|W(G)/W(H_0)|$ is a power of $2$.  As an application, one can show that if $H$ is
a non--connected proper closed subgroup of $\SO(3)$ with $H_0=\SO(2)$, then $H$ is
isomorphic to $\OO(2)$.  A proof may use the fact that $H$ is $2$--toral, and that a
maximal $2$--toral subgroup in $H$ is $2$--stubborn \cite{JMO}.  A $2$--compact version
of this result also holds.  Suppose $X$ is a $2$--compact group such that
there are two monomorphisms of $2$--compact groups
$B\SO(2)\2 \lra BX$ and $BX \lra B\SO(3)\2$.  Then, along the line of a similar argument,
one can also show that $BX$ is homotopy equivalent to $B\OO(2)\2$ if $X$ is not connected.
In the case of $X$ being connected, the classifying space $BX$ is either
$B\SO(2)\2$ or $B\SO(3)\2$.

In \fullref{thm2}, Lie groups of type $(C_2, A_1 \times A_1)$ has been discussed.  An example is given
by $\Sp(1) \times \Sp(1) \subset (\Sp(1) \times \Sp(1))\rtimes \Z /2
\subset \Sp(2)$.  The middle
group can be regarded as the wreath product $\Sp(1)\smallint \Sigma_n$ for $n=2$.  We ask
for what $n$ and $p$ its classifying space is $p$--compact.  Note that $\Sp(1)
\smallint \Sigma_n$
is a proper closed subgroup of $\Sp(n)$.

\begin{prop}
Let $\Gamma (n)$ denote the wreath product $\Sp(1)\smallint \Sigma_n$.  Then
$$
\mathbb{P} (B\Gamma (n)) =
\begin{cases}
\Pi &\text{if } n=2\\
\{ p \in \Pi ~|~ p>n \}  &\text{if } n\ge 3
\end{cases}
$$
\end{prop}

\begin{proof}
When $n=2$, the desired result has been shown in our proof of the part (b)
of \fullref{thm2}.  Recall
from \cite{It} that if $B\Gamma (n)$ is $p$--compact, then $\pi_0B\Gamma (n)=\Sigma_n$ must be
$p$--nilpotent.  For $n \ge 4$, it follows that $\Sigma_n$ is $p$--nilpotent if and only if $p>n$.
Since the group $\Gamma (n)$ includes the normalizer of a maximal torus of
$\Sp(n)$, we see 
$B\Gamma (n) \simeq_{p} BSp(n)$ if $p>n$.  Thus $\mathbb{P} (B\Gamma (n)) =\{ p \in \Pi \ | \ p>n \}$
for $n \ge 4$.

For $n=3$, note that $\Sigma_3$ is $p$--nilpotent if and only if $p\ne 3$.  So it remains to 
prove that $B\Gamma (3)$ is not $2$--compact.  We consider a subgroup $H$ of $\Gamma (3)$
which makes the following diagram commutative: 
$$
\CD
\displaystyle \prod_{}^3Sp(1) @= \displaystyle \prod_{}^3Sp(1) @>>> * \\
 @VVV @VVV @VVV              \\
H @>>> \Gamma (3) @>>> \Z /2 \\
 @VVV @VVV @|                 \\
\Z /3 @>>> \Sigma_3 @>>> \Z /2
\endCD
$$
The fibration $BH \lra B\Gamma (3) \lra B\Z /2$ is preserved by the
completion at $p=2$.  Hence, if $B\Gamma (3)$ were $2$--compact, the
space $\Omega (BH)\2$ would be a connected $2$--compact group so that the
cohomology $H^*(BH ; \Q \2)$ should be a polynomial ring, (see Dwyer and
Wilkerson~\cite[Theorem~9.7]{DWb}).  Though $H^*\bigl(B\prod_{}^3Sp(1)
; \Q \2\bigr)$ is a polynomial ring, its invariant ring $H^*(BH;\Q\2)=
H^*\bigl(B\prod_{}^3Sp(1) ; \Q \2\bigr)^{\Z /3}$ is not a polynomial
ring, since the group $\Z /3$ is not generated by reflections.
This contradiction completes the proof.
\end{proof}

For $(G, H)=(\Sp(n), \Sp(1)\smallint \Sigma_n)$, we note that $(G, H_0)$
is a type of $(C_n, A_1 \times \cdots \times A_1)$.  This is one of the
cases that the Weyl group $W(H_0)$ is a normal subgroup of $W(G)$ (see
Asano~\cite{A}) discussed in our proof of the part (a) of \fullref{thm2}.
Finally we talk about the only remaining case $(G, H_0)=(F_4, D_4)$.
An example is given by $\Spin(8) \subset \Spin(8)\rtimes \Sigma_3 \subset
F_4$.  Let $\Gamma$ denote the middle group $\Spin(8)\rtimes \Sigma_3$.
Then we can show that $\mathbb{P} (B\Gamma ) =\{ p \in  \Pi \ | \ p>3 \}$.
To show that $B\Gamma $ is not $2$--compact, one might use the fact,
(see Adams~\cite[Theorem~14.2]{Ad}), that $W(F_4)=W(\Spin(8)) \rtimes
\Sigma_3$, and that its subgroup $W(\Spin(8)) \rtimes \Z /3$ is not a
reflection group.

\bibliographystyle{gtart}
\bibliography{link}

\end{document}